\documentclass[12pt,reqno,a4wide]{amsart}

\usepackage{tikz} 
\usepackage{calrsfs}
\usepackage{mathrsfs}

\usepackage{array}
\usepackage{hyperref}

\usetikzlibrary{decorations.pathreplacing}
\usetikzlibrary{fit}

\usepackage{pgfplots}

\allowdisplaybreaks

\newtheorem{theorem}{Theorem}

\newtheorem{prop}{Proposition}

\catcode`,\active

\catcode`\,12

\begin{document}
	\title
{Fibonacci numbers along residue classes and convolutions}

\author[H. Prodinger ]{Helmut Prodinger }
\address{Department of Mathematics, University of Stellenbosch 7602, Stellenbosch, South Africa
	and
	NITheCS (National Institute for
	Theoretical and Computational Sciences), South Africa.}
\email{warrenham33@gmail.com}


\begin{abstract} The sequence $F_{dn+h}$ and its convolutions have (for $h=0$) been studied in a recent paper at the arxiv.
	The instance with general $h$ is more involved and uses Chebyshev polynomials. 
\end{abstract}

\maketitle

\section{A review of the Chebyshev polynomials of the second kind}

These polynomials might be defined and studied in many ways \cite{AAR}, but we consider $U_n(x)$ via the generating function:
\begin{equation*}
\frac{1}{1-2tx+t^2}=\sum_{n\ge0}U_n(x)t^n.
\end{equation*}
The denominator $1-2tx+t^2$ leads to the recursion $U_n(x)-2xU_{n-1}(x)+U_{n-2}(x)=0$, with $U_0(x)=1$ and $U_1(x)=2x$.
From this one can obtain via induction the explicit formula
\begin{equation*}
\begin{cases}
	U_{2n}(x)=\sum_kx^{2k}(-1)^{n-k}\binom{n+k}{2k}2^{2k},\\
		U_{2n+1}(x)=\sum_kx^{2k+1}(-1)^{n-k}\binom{n+k+1}{2k+1}2^{2k+1}.
	\end{cases}
\end{equation*}
Since
\begin{equation*}
\frac{\partial}{\partial x}\frac{1}{1-2tx+t^2}=\frac{2t}{(1-2tx+t^2)^2} \quad\text{and}\quad\frac{\partial}{\partial x}\frac{1}{(1-2tx+t^2)^{s}}=\frac{2st}{(1-2tx+t^2)^{s+1}},
\end{equation*}
we may set
\begin{equation*}
\frac{1}{(1-2tx+t^2)^{s+1}}=\sum_{n\ge0}U_n^{(s)}(x)t^n
\end{equation*}
and find again explicit formul\ae\ 
\begin{equation*}
\begin{cases}
	U_{2n}^{(s)}(x)=\sum_kx^{2k}(-1)^{n-k}\binom{n+k+s}{2k+s}\binom{2k+s}{s}2^{2k},\\
	U_{2n+1}^{(s)}(x)=\sum_kx^{2k+1}(-1)^{n-k}\binom{n+k+s+1}{2k+s+1}\binom{2k+s+1}{s}2^{2k+1}.
\end{cases}
\end{equation*}
This may be derived via induction on the parameter $s$.

A small variation is via the introduction of a $\pm$ sign; let $\epsilon\in\{-1,+1\}$ and
\begin{equation*}
	\frac{1}{(1-2tx-\epsilon t^2)^{s+1}}=\sum_{n\ge0}\overline{U}_n^{(s)}(x)t^n
\end{equation*}
then
\begin{equation*}
	\begin{cases}
		\overline{U}_{2n}^{(s)}(x)=\sum_kx^{2k}\epsilon^{n-k}\binom{n+k+s}{2k+s}\binom{2k+s}{s}2^{2k},\\
		\overline{U}_{2n+1}^{(s)}(x)=\sum_kx^{2k+1}\epsilon^{n-k}\binom{n+k+s+1}{2k+s+1}\binom{2k+s+1}{s}2^{2k+1}.
	\end{cases}
\end{equation*}

\section{The $d$-section of Fibonacci numbers}
The recent paper \cite{ukrain} considers for Fibonacci numbers $F_{n+2}=F_{n+1}+F_n$, $F_0=0$, $F_1=1$
\begin{equation*}
\sum_{n\ge0}F_{2n}z^n=\frac{z}{1-3z+z^2} \quad\text{and, more generally,}\quad \sum_{n\ge0}F_{dn}z^n.
\end{equation*}
Here, we are interested in
\begin{equation*}
	\sum_{n\ge0}F_{dn+h}z^n
\end{equation*}
which makes the analysis more complete and also more challenging.

\begin{prop}

\begin{equation*}
\sum_{n\ge0}F_{nd+h}z^n=\frac{F_h+(-1)^hzF_{d-h}}{1-zL_{d}+(-1)^dz^2}.
\end{equation*}
\end{prop}
\textbf{Proof.} We will use the Binet formul\ae 
\begin{equation*}
F_n=\frac{\alpha^n-\beta^n}{\alpha-\beta},\ L_n=\alpha^n+\beta^n
\quad\text{with}\quad \alpha=\frac{1+\sqrt5}{2},\ \beta=\frac{1-\sqrt5}{2}.
\end{equation*}
Then
\begin{align*}
X&:=(1-zL_{d}+(-1)^dz^2)	\sum_{n\ge0}F_{nd+h}z^n\\
&=(1-z(\alpha^{d}+\beta^{d})+(-1)^dz^2)	\sum_{n\ge0}\frac{\alpha^{nd+h}-\beta^{nd+h}}{\alpha-\beta}z^n\\
&=\frac{1-z(\alpha^{d}+\beta^{d})+(-1)^dz^2}{\alpha-\beta}	\sum_{n\ge0}(\alpha^{nd+h}-\beta^{nd+h})z^n\\
&=\frac{1-z(\alpha^{d}+\beta^{d})+(-1)^dz^2}{\alpha-\beta}\bigg[\frac{\alpha^{h}}{1-z\alpha^{d}}- \frac{\beta^{h}}{1-z\beta^{d}}\bigg]\\
&=\frac{1-z(\alpha^{d}+\beta^{d})+(-1)^dz^2}{\alpha-\beta}\frac{\alpha^{h}(1-z\beta^{d})- \beta^{h}(1-z\alpha^{d})}{1-z(\alpha^{d}+\beta^d)+(-1)^dz^2}\\
&=\frac{\alpha^{h}(1-z\beta^{d})- \beta^{h}(1-z\alpha^{d})}{\alpha-\beta}\\
&=F_h-z\frac{\alpha^{h}(\beta^{d})- \beta^{h}(\alpha^{d})}{\alpha-\beta}\\
&=F_h-z(-1)^{h}\frac{\beta^{d-h}- \alpha^{d-h}}{\alpha-\beta}\\
&=F_h+z(-1)^{h}F_{d-h}.\qed
\end{align*}

The $s$-fold convolution of the sequence $\langle F_{nd+h}\rangle_n$ is (in \cite{ukrain}for $h=0$ only) defined via the generating function
\begin{equation*}
\biggl(\sum_{n\ge0}F_{nd+h}z^n\biggr)^{s+1}.
\end{equation*}
Indeed, for $h=0$, we have
\begin{equation*}
\frac{F_h+(-1)^hzF_{d-h}}{1-zL_{d}+(-1)^dz^2}\bigg|_{h=0}=\frac{zF_{d}}{1-zL_{d}+(-1)^dz^2},
\end{equation*}
and the simpler numerator makes life a bit easier. We have to deal with 
\begin{equation*}
(F_h+(-1)^hzF_{d-h})^{s+1}=\sum_{j=0}^{s+1}\binom{s+1}{j}F_h^{s+1-j}(-1)^{jh}z^jF_{d-h}^{j}
\end{equation*}
and eventually with
\begin{equation*}
\sum_{j=0}^{s+1}\binom{s+1}{j}F_h^{s+1-j}(-1)^{jh}z^jF_{d-h}^{j}\frac{1}{(1-zL_{d}+(-1)^dz^2)^{s+1}}
\end{equation*}
Fortunately, following the introduction, the expansion of
\begin{equation*}
\frac{1}{(1-zL_{d}+(-1)^dz^2)^{s+1}}=\sum_{n\ge0}\overline{U}_n^{(s)}(L_d)z^n
\end{equation*}
is known, and $\epsilon=(-1)^{d-1}$.

Therefore we have the fully explicit $s$-fold convolution of the sequence $\langle F_{nd+h}\rangle_n$ via its generating function:
\begin{theorem}
	\begin{equation*}
		\biggl(\frac{F_h+(-1)^hzF_{d-h}}{1-zL_{d}+(-1)^dz^2}\biggr)^{s+1}=
		\sum_{j=0}^{s+1}\binom{s+1}{j}F_h^{s+1-j}(-1)^{jh}z^jF_{d-h}^{j}
		\sum_{n\ge0}\overline{U}_n^{(s)}(L_d)z^n.
	\end{equation*}
Expanded,
	\begin{equation*}
	\biggl(\frac{F_h+(-1)^hzF_{d-h}}{1-zL_{d}+(-1)^dz^2}\biggr)^{s+1}=\sum_{n\ge0}z^n
	\sum_{j=0}^{s+1}\binom{s+1}{j}F_h^{s+1-j}(-1)^{jh}F_{d-h}^{j}
\overline{U}_{n-j}^{(s)}(L_d).\qed
\end{equation*}

\end{theorem}
The special case $h=0$ leads to
\begin{equation*}
\sum_{n\ge0}z^n
	F_{d}^{s+1}	\overline{U}_{n-j}^{(s)}(L_d).
\end{equation*}

\textbf{Remark.} The $d$-section of Lucas numbers can also be considered:
\begin{equation*}
	\sum_{n\ge0}L_{nd+h}z^n=\frac{L_h+(-1)^{h-1}zL_{d-h}}{1-zL_{d}+(-1)^dz^2}.
\end{equation*}
The details are left to the interested reader; the denominator is the same, so only the numerator needs to be changed.

\bibliographystyle{plain}

\end{document}